\documentclass{article}

\usepackage{fullpage}
\usepackage{amsmath,amsthm,amssymb}
\usepackage{amsfonts}
\usepackage{graphicx}
\usepackage{subcaption} 
\usepackage{placeins}
\usepackage{floatrow}
\usepackage{wasysym}
\makeatletter
\newtheorem*{rep@theorem}{\rep@title}
\newcommand{\newreptheorem}[2]{%
\newenvironment{rep#1}[1]{%
 \def\rep@title{#2 \ref{##1}}%
 \begin{rep@theorem}}%
 {\end{rep@theorem}}}
\makeatother
\theoremstyle{plain}
\newtheorem{theorem}{Theorem}
\newreptheorem{theorem}{Theorem} 
\newtheorem{lemma}[theorem]{Lemma}

\theoremstyle{definition}
\theoremstyle{definition}
\newtheorem{definition}[theorem]{Definition}

\newcounter{example}

\usepackage{tikz}
\usetikzlibrary{shapes, arrows, calc, positioning}

\usepackage[shortlabels]{enumitem} 

\usepackage{multicol}
\usepackage[mathscr]{eucal}
\usepackage{subcaption}

\usepackage{color}
\usepackage{marginnote}

\usepackage{graphicx}

\renewcommand{\bar}{\overline}

\tikzstyle{vertex}=[draw,thick,fill=white,circle,inner sep=2pt]

\newcommand{\modstrong}{\ \raisebox{1pt}{\rotatebox[origin=c]{90}{\Bowtie}}\;}
\newcommand{\strongprod}{\boxtimes}
\newcommand{\boxprod}{\square}

\definecolor{codegreen}{rgb}{0,0.4,0}
\definecolor{codegray}{rgb}{0.2,0.2,0.2}
\definecolor{codered}{rgb}{0.6,0,0}
\definecolor{backcolour}{rgb}{0.9,0.9,0.9}
\definecolor{black}{rgb}{0,0,0}

\title{Two Distinct Eigenvalues from a New Graph Product}
\author{Eric Culver\footnote{Brigham Young University, Provo, UT, \emph{eric.culver@mathematics.byu.edu}} and Mark Kempton\footnote{Brigham Young University, Provo, UT, \emph{mkempton@mathematics.byu.edu}}}
\date{}

\begin{document}

\maketitle

\begin{abstract}
    The parameter $q(G)$ of a graph $G$ is the minimum number of distinct eigenvalues of a symmetric matrix whose pattern is given by $G$.  We introduce a novel graph product by which we construct new infinite families of graphs that achieve $q(G)=2$.  Several graph families for which it is already known that $q(G)=2$ can also be thought of as arising from this new product.
\end{abstract}

\section{Introduction}

Graph inverse eigenvalue problems have been the subject of extensive research for many years now, and most recently considerable attention has been given to the parameter $q(G)$ of a graph $G$, which is the minimum number of distinct eigenvalues that can be achieved by a symmetric matrix whose pattern is given by a graph.  More specifically, for a graph $G=(V(G),E(G))$, we define \[\mathcal{S}(G)=\{\text{symmetric matrices } A=[a_{ij}] \mid a_{ij}=0 \text{ for }i\neq j \iff ij \not\in E(G)\}.\]  Note that no restriction is put on the diagonal entries of $A$. The \emph{Inverse Eigenvalue Problem for Graphs (IEPG)} asks what spectra can be achieved by matrices $A\in\mathcal{S}(G)$ given the graph $G$ (see \cite{barrett2020inverse,hogben2022inverse}).  This is a difficult problem in general, and most work in the area addresses subquestions such as the minimum rank, maximum nullity, and maximum multiplicity problems \cite{adm2019achievable,aim2008zero,hogben2022inverse}.

One important subproblem of the IEPG is to understand the possible multiplicities of eigenvalues of matrices in $\mathcal{S}(G)$.  This entails understanding the possible numbers of distinct eigenvalues of matrices in $\mathcal{S}(G)$.  To this end, for a symmetric matrix $A$, we first define $q(A)$ to be the number of distinct eigenvalues of $A$, and then
\[
q(G) = \min\{q(A) \mid A\in\mathcal{S}(G)\}.
\]

The $q(G)$ parameter was introduced in \cite{leal2002minimum} and has been extensively studied.  In general, determining $q(G)$ can still be quite challenging, but many results have been found \cite{ahmadi2013minimum,barrett2015generalizations,leal2002minimum,hogben2022inverse,levene2019nordhaus,levene2022orthogonal}.  Note that it is easy to characterize graphs for which $q(G)=1$: since the matrices in $\mathcal{S}(G)$ are symmetric, then $q(A)=1$ if and only if $A$ is a multiple of the identity, which implies that $q(G)=1$ if and only if $G$ has no edges.  The problem becomes much more subtle already when we investigate graphs which achieve exactly two distinct eigenvalues.  Graphs for which $q(G)=2$ have turned out to be interesting and difficult to characterize \cite{ahmadi2013minimum,barrett2023sparsity}.  Families of graphs with $q(G)=2$ have been constructed and studied in \cite{barrett2023regular,barrett2024graphs}. Of note, \cite{ahmadi2013minimum} implies that any connected graph can appear as an induced subgraph of a graph with $q(G)=2$ making a complete characterization of these graphs challenging (see also \cite{abiad2023bordering}).  

In this note, we attempt to unify some of the families of graphs with $q(G)=2$ by viewing them as resulting from graph products. Some of our constructions will involve the well-known strong product of graphs $G\strongprod H$.  We will also as introduce a new graph product, a variant of the strong product, which we call the \emph{modified strong product}, which we will denote using $G\modstrong H$. See Definition \ref{def:modstrong} below. Matrices whose pattern corresponds to strong products and modified strong products can be described easily using tensor products of matrices for the factors.  Using this we construct new families of graphs with $q(G)=2$, as well as view some known examples with $q(G)=2$ as arising from this new product.  A surprising connection with the chromatic index of a graph will be explored.  Specifically, for our main results, we will prove that all of the following families of graphs can achieve 2 distinct eigenvalues (see Theorems \ref{thm:onefactor}, \ref{thm:maxdeg}, and \ref{thm:complete} below):
\begin{itemize}
    \item $G\modstrong K_k$ where $G$ is $k$ regular with chromatic index $k$.
    \item $G\strongprod K_{k+1}$ where $G$ has maximum degree $k$.
    \item $G\strongprod K_c$ where $G$ arises from a linear hypergraph with maximum degree $k$ and chromatic index $c$ with $c>k$.
    \item $G\strongprod K_{c+1}$ where $G$ is as above, except $k=c$.
    \item $G\modstrong K_c$ in the same context, but when the hypergraph is $k$-regular, and again, $c=k$.
\end{itemize}
Each of these theorems is proven by decomposing the edges sets of the graphs into partitions where each piece has 2 eigenvalues, and then examining the tensor product of the corresponding 0-1 matrices with certain matrices arising from orthogonal matrices.

The rest of the paper will be organized as follows.  In Section \ref{sec:prod} we will introduce the modified strong product and investigate the relevant spectral properties.  In Section \ref{sec:coloring} prove the main results and describe examples achieving $q(G)=2$ from constructions using the modified strong product and the strong product.  We will end with some discussion of open questions and future directions.

\section{The modified strong product}\label{sec:prod}

We will begin with the definition of the modified strong product.  We first remind the reader of the definition of the strong product of two graphs.  Recall that the strong product $G\strongprod H$ is the graph whose vertex set is $V(G)\times V(H)$ containing all edges of the form 
\begin{itemize}
        \item $(g,h)(g,h')$ where $hh'\in E(H)$,
        \item $(g,h)(g',h)$ where $gg' \in E(G)$,
        \item $(g,h)(g',h')$ where $gg' \in E(G)$ and $hh' \in E(H)$.
    \end{itemize}
The adjacency matrix of $G\strongprod H$ is $(A_G+I)\otimes (A_H+I)$, where $\otimes$ denotes the Kronecker product of matrices, i.e.~ for an $m\times n$ matrix $M$ and $p\times q$ matrix $N$, $M\otimes N$ is the $np\times nq$ matrix whose entries can be thought of in block form, where the entries of the $i,j$ block are $m_{i,j}N$

\begin{definition}\label{def:modstrong}
    The \emph{modified strong product} of $G$ and $H$, which we will denote by $G \modstrong H$, is defined as the graph on vertex set $V(G)\times V(H)$ and whose edges are:
    \begin{itemize}
        \item $(g,h)(g',h)$ where $gg' \in E(G)$,
        \item $(g,h)(g',h')$ where $gg' \in E(G)$ and $hh' \in E(H)$.
    \end{itemize}
    \end{definition}
    Note that this is very similar to the strong product of $G$ and $H$, but we are not including the edges $(g,h)(g,h')$ for $hh' \in E(H)$. In particular, this means that $G \modstrong H$ is not necessarily isomorphic to $H \modstrong G$.  See Figure \ref{fig:modstrong}.

    Note that the adjacency matrix of $G \modstrong H$ is $A_G \otimes (A_H + I)$.

 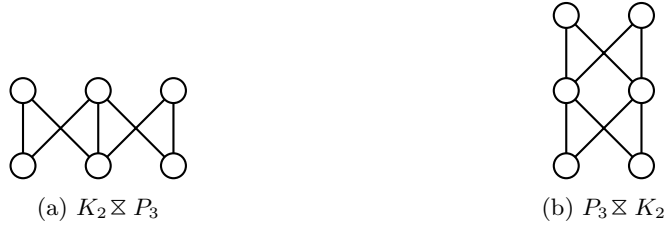
\begin{figure}
     \begin{subfigure}{0.4\textwidth}
          \centering
        \begin{tikzpicture}
        [
                vert/.style={circle,fill=white,draw=black},
                 edge/.style={thick},
             ]
          \draw[edge]   (0,0)node[vert]{}--(0,1)node[vert]{}--(1,0)node[vert]{}--(1,1)node[vert]{}--(2,0)node[vert]{}--(2,1)node[vert]{} (0,0)--(1,1) (1,0)--(2,1);
         \end{tikzpicture}
         \caption{$K_2\modstrong P_3$}
    \end{subfigure}
     \begin{subfigure}{0.4\textwidth}
          \centering
     \begin{tikzpicture}
        [
                vert/.style={circle,fill=white,draw=black},
                 edge/.style={thick},
             ]
          \draw[edge]   (0,0)node[vert]{}--(0,1)node[vert]{}--(0,2)node[vert]{}(1,0)node[vert]{}--(1,1)node[vert]{}--(1,2)node[vert]{} (0,0)--(1,1)--(0,2) (1,0)--(0,1)--(1,2);
         \end{tikzpicture}
         \caption{$P_3\modstrong K_2$}   
    \end{subfigure}
     \caption{Two examples of modified strong products.  Note in particular that $K_2\modstrong P_3$ is not isomorphic to $P_3\modstrong K_2$. }
     \label{fig:modstrong}
 \end{figure}

\begin{definition}
    For some index set $K \subseteq [k]$, let $D_K$ be the diagonal matrix:
    \[ d_{ij} = \begin{cases} 1 & \text{ if $i = j \in K$} \\ 0 & \text{ otherwise} \end{cases} \]
    For ease of notation, let $D_i = D_{\{i\}}$.
\end{definition}

\begin{lemma}
    \label{lem:existence}
    There exists an orthogonal $k$ by $k$ matrix $Q$ such that $Q D_K Q^T$ is a matrix with all nonzero entries for all $K \subseteq [k]$ except for $K = \emptyset, [k]$.
\end{lemma}
\begin{proof}
    We will exhibit how to construct such an orthogonal matrix $Q$ for any size $k$.  We will obtain a $Q$ with the desired properties from a \emph{Householder matrix} \[Q=H(u):=I-2uu^T\] where $u$ is a unit vector.  It is easy to see that for any unit vector $u$, $H(u)H(u)^T = H(u)^2=I$, so this always yields an orthogonal matrix.  We will show that $u$ can be chosen in such a way to yield a $Q$ with the desired property that $QD_KQ^T$ is nowhere 0 for every nonempty proper subset $K$ of $[k]$.  

    Let $K$ be a nonempty proper subset of $[k]$ and given a unit vector $u$, let $u_K$ be the vector indexed by the entries of $K$, and $u_{\bar K}$ the remaining entries of the vector. By reordering the entries if necessary, we may assume without loss of generality the elements of $K$ are indexed first.  Then \[D_K=\begin{bmatrix}
        I_K&0\\0&0
    \end{bmatrix},\] and a straightforward computation yields
    \[
    H(u)D_KH(u)^T=\begin{bmatrix}
        I_K-4(1-u_K^Tu_K)u_Ku_K^T & -2(1-2u_ K^Tu_K)u_Ku_{\bar K}^T\\-2(1-2u_K^Tu_K)u_{\bar K}u_K^T&4(u_K^Tu_K)u_{\bar K}u_{\bar K}^T
    \end{bmatrix}.
    \]
    We want this to be nowhere zero for every nonempty, proper subset $K$.  The first requirement we impose is that $u$ be a unit vector with no 0 entries.  This will guarantee that the matrices $u_Ku_K^T$, $u_Ku_{\bar K}^T$, $u_{\bar K}u_K^T$ and $u_{\bar K}u_{\bar K}^T$ will have no entries equal to 0 no matter what $K$ is.  It will also guarantee that $1-u_K^Tu_K$ is nonzero for any proper subset $K$.  Thus, $QD_KQ^T$ with have no zero entries as long as we can choose a unit vector $u$ with no zero entries that satisfies:
    \begin{itemize}
       \item $u_K^Tu_K \neq \frac12$ for \emph{every} nonempty proper subset $K$; and
        \item no diagonal entry of $(1-u_K^Tu_K)u_Ku_K^T$ is equal to $\frac14$ for every nonempty proper subset $K$.
    \end{itemize}
    There are likely many choice of a nowhere 0 unit vector $u$ satisfying these properties; we need exhibit only 1 for each value of the dimension $k$.  
    
    For $k$ odd, taking $u$ to be the constant unit vector with every entry equal to $1/\sqrt{k}$ will do. For any $K$, $u_KTu_K$ will be an integer divided by the odd number $k$, so this can never equal $1/2$, giving us the first property.  Then since $1-u_K^Tu_K$ is strictly smaller than 1, the diagonal entries of $(1-u_K^Tu_K)u_Ku_K^T$ are strictly smaller than the diagonal entries of $u_K^Tu_K$, which are all $\frac1k$, which is smaller than $\frac14$ unless $k=1,3$, and these cases can be checked to work directly.  Thus the second condition is satisfied.

    For $k$ even, we do an analogous construction, but now we take $u$ to be the vector with a 2 in one entry, 1 in all the other entries, and then normalize.  So the entries of $u$ are $2/\sqrt{k+3}$ occurring once, and $1/sqrt{k+3}$ occurring $k-1$ times.  For even $k$, $k+3$ is odd, so the proof will proceed precisely as in the odd case.   
\end{proof}

Let $Q$ be defined by the lemma above, and then let $J_i = Q D_i Q^T$ for all $i \in [k]$. Then, by the lemma, we have that any nonempty sum of the $J_i$s will have all nonzero entries except for the sum of all of them which is equal to the identity.   We can also see that these $J_i$s satisfy the following basic properties:

 \begin{lemma}
     \label{lem:properties}
     For all $i \ne j \in [k]$:
     \begin{align*}
         J_i^2 &= J_i \\
         J_i J_j &= 0
     \end{align*}
    
     Let the $i$th column of $Q$ be $q_i$. 
     Then $q_i, q_j$ are eigenvectors of $J_i$ with eigenvalues $1,0$ respectively.
 \end{lemma}
 \begin{proof}
     \begin{align*}
         J_i^2 &= (Q D_i Q^T)(Q D_i Q^T) \\
         &= Q D_i^2 Q^T \\
         &= Q D_i Q^T \\
         &= J_i
     \end{align*}
     Also:
     \begin{align*}
         J_i J_j &= (Q D_i Q^T)(Q J_j Q^T) \\
         &= Q D_i D_j Q^T \\
         &= Q 0 Q^T \\
         &= 0
     \end{align*}
     Let $e_i$ be the vector with one in the $i$th component and zeros elsewhere. Since $Qe_i = q_i$, we see that $Q^T q_i = e_i$.
     Then we can see that:
     \begin{align*}
         J_i q_i &= Q D_i Q^T q_i \\
         &= Q D_i e_i \\
         &= Q e_i \\
         &= q_i
     \end{align*}
     Meaning $q_i$ is an eigenvector of $J_i$ with eigenvalue $1$. 
     Similarly:
     \begin{align*}
         J_i q_j &= Q D_i Q^T q_j \\
         &= Q D_i e_j \\
         &= Q 0 \\
         &= 0
     \end{align*}
     Meaning $q_j$ is an eigenvector of $J_i$ with eigenvalue $0$.
 \end{proof}

Using this we can prove the main lemma which we will use for our results.

\begin{lemma}
    \label{lem:eigenvals}
    Suppose we have $k$ $n\times n$ matrices $A_1, \ldots, A_k$ which have the following eigenvalues
    \[ \lambda_{1,1} \ldots, \lambda_{1,n}, \ldots, \lambda_{k,1} \ldots, \lambda_{k,n} \] 
    and corresponding eigenvectors 
    \[ v_{1,1}, \ldots, v_{1,n}, \ldots, v_{k,1}, \ldots, v_{k,n}. \]
    Let $M$ be the matrix:
    \[ (A_1 \otimes J_1) + \cdots + (A_k \otimes J_k). \]
    Then the the eigenvalues of $M$ are:
    \[ \lambda_{1,1}, \ldots, \lambda_{1,n}, \ldots, \lambda_{k,1} \ldots, \lambda_{k,n} \]
    with corresponding eigenvectors:
    \[ v_{1,1} \otimes q_1, \ldots, v_{1,n} \otimes q_1, \ldots, v_{k,1} \otimes q_k, \ldots, v_{k,n} \otimes q_k \]
    where $q_1, \ldots, q_k$ are the columns of $Q$.
\end{lemma}
\begin{proof}
    We can see that:
    \begin{align*}
        (A_i \otimes J_i)(v_{i,l} \otimes q_i) &= A_i v_{i,l} \otimes J_i q_i \\
        &= \lambda_{i,l} v_{i,l} \otimes q_i \\
        &= \lambda_{i,l} (v_{i,l} \otimes q_i)
    \end{align*}
    Meaning that $v_{i,l} \otimes q_i$ is an eigenvector of $A_i \otimes J_i$ with eigenvalue $\lambda_{i,j}$.

    Also, for $j\neq i$, we can see that:
    \begin{align*}
        (A_i \otimes J_i)(v_{j,l} \otimes q_j) &= A_i v_{j,l} \otimes J_i q_j \\
        &= \lambda_{i,l} v_{j,l} \otimes 0 \\
        &= 0
    \end{align*}
    Meaning that $v_{j,l} \otimes q_j$ is an eigenvector of $A_i \otimes J_i$ with eigenvalue $0$.

    Therefore:
    \begin{align*}
        M(v_{i,l} \otimes q_i) &= ((A_1 \otimes J_1) + \cdots + (A_k \otimes J_k))(v_{i,l} \otimes q_i) \\ 
        &= (A_1 \otimes J_1)(v_{i,l} \otimes q_i) + \cdots + (A_i \otimes J_i)(v_{i,l} \otimes q_i) + \cdots + (A_k \otimes J_k)(v_{i,l} \otimes q_i) \\
        &= 0 + \cdots + \lambda_{i,l} v_{i,l} \otimes q_i + \cdots + 0 \\
        &= \lambda_{i,l} v_{i,l} \otimes q_i
    \end{align*}
    Which shows that $v_{i,l} \otimes q_i$ is an eigenvector of $M$ with eigenvalue $\lambda_{i,l}$.
\end{proof}

The next lemma addresses the pattern of our matrices.

\begin{lemma}
    \label{lem:blocks}
    If $A_1, \ldots, A_k$ are $n\times n$ matrices with 0,1 entries, then the pattern of
    \[ M = (A_1 \otimes J_1) + \cdots + (A_k \otimes J_k) \]
    is determined by the entries of:
    \[ A = A_1 + \cdots + A_k. \]
    Let the entries of $A$ be given by $A=[a_{i,j}]_{i,j=1}^n$.
    Specifically, $M$ is an $nk\times nk$ matrix with blocks of size $k$ by $k$. The $l,m$ block of $M$ will have the pattern:
    \[ \begin{cases} 0 & \text{ if $a_{l,m} = 0$} \\ J & \text{ if $0 < a_{l,m} < k$} \\ I & \text{ if $a_{l,m} = k$} \end{cases} \]
\end{lemma}
\begin{proof}
    Since each $A_i$ is a 0,1 matrix, the matrices $A_i \otimes J_i$ will be block matrices with blocks equal to $J_i$ or $0$. Therefore, the matrix $M$ will consist of blocks which are sums of the $J_i$. Since the pattern of a sum of $J_i$ is uniquely determined by the number of $J_i$ summed, we can determine the pattern of the blocks of $M$ from the sum of the $A_i$. 
\end{proof}

\section{Chromatic index and products with cliques}\label{sec:coloring}


Recall that the chromatic index of a graph $G$, denoted $\chi'(G)$, is the minimum number of colors needed to give a proper edge coloring of the graph; i.e. no two edges of the same color share a vertex.  Note that subgraph of $G$ given by all the edges of a fixed color in a proper edge coloring gives a matching, or 1-factor, in the graph.  A $k$-regular graph with chromatic index $k$ is sometimes called 1-factorable, meaning its edge set can be written as the disjoint union of $k$ sets of edges (the factors), where no edges within a factor share a vertex. 

\begin{theorem}
    \label{thm:onefactor}
    If $G$ is 1-factorable, i.e., if $G$ is $k$-regular with chromatic index $k$, then $q(G \modstrong K_k)=2$.
\end{theorem}

\begin{proof}
    Let $G_i$ be the graph of the $i$th factor, and let $A_i$ be its adjacency matrix. 
    Note that the eigenvalues of each $A_i$ are $1$ and $-1$, each with multiplicity $n/2$ where $n$ is the total number of vertices. 
    Let $M$ be the matrix:
    \[ M = (A_1 \otimes J_1) + \cdots + (A_k \otimes J_k) \]
    By Lemma \ref{lem:eigenvals}, the eigenvalues of $M$ will be $-1,1$ with multiplicity.

    We can also see that $M$ will have the pattern of the adjacency matrix of $G \modstrong K_k$. Using Lemma \ref{lem:blocks}, the pattern of $M$ is determined by $A_1 + \cdots + A_k = A$ which is equal to the adjacency matrix of $G$. By that lemma, $M$ will have a block of all nonzero elements corresponding to a one in the adjacency matrix of $G$, and a block of zeros corresponding to a zero in the adjacency matrix of $G$. This is equivalent to saying that the pattern of $M$ is equal to the matrix $A \otimes J$. Since $J$ can be thought of as $I$ plus the adjacency matrix of $K_k$, this can be seen to be the adjacency matrix of $G \modstrong K_k$.

    Therefore, a symmetric matrix with the same pattern as the adjacency matrix of $G \modstrong K_k$ has only two eigenvalues, $-1,1$, meaning that $q(G \modstrong K_k)=2$. 
\end{proof}

\begin{theorem}
    \label{thm:maxdeg}
    If connected $G$ has max degree $k$, then $q(G \strongprod K_{k+1})=2$.
\end{theorem}
\begin{proof}
    Since $G$ has max degree $k$, by Vizing's theorem, it has a $k+1$-coloring of the edges.
    This coloring will be such that for any vertex $v$, at least one color will miss that vertex.
    Let $G_i$ be the subgraph of $G$ on the same vertex set but only containing the edges of color $i$.
    Construct matrices $A_1, \ldots, A_{k+1}$ such that $A_i$ has entries:
    \[ A_{i,l,m} = \begin{cases} 1 & \text{ if $l = m$ and $l$ is not incident to any edges of color $i$} \\ 1 & \text{ if $l$ adjacent to $m$ by edge of color $i$} \\ 0 & \text{ otherwise} \end{cases} \]
    
    Define $M = (A_1 \otimes J_1) + \cdots + (A_{k+1} \otimes J_{k+1})$.
    Notice that each $A_i$ has eigenvalues $1$ or $-1$ by construction.
    Therefore, by Lemma \ref{lem:eigenvals}, the eigenvalues of $M$ will be $-1,1$ with some multiplicities.

    We can also see that $M$ will have the pattern of the adjacency matrix of $G \strongprod K_{k+1}$.
    Using Lemma \ref{lem:blocks}, the pattern of $M$ is determined by $A_1 + \cdots + A_k = B$. The off-diagonal entries of $B$ correspond exactly to the adjacency matrix of $G$. Therefore, $M$ will have a block of all nonzero elements corresponding to a one in the adjacency matrix of $G$, and a block of zeros corresponding to a zero in the adjacency matrix of $G$. Since every vertex of $G$ fails to be incident to some color, the diagonal entries of $B$ are nonzero. Since $G$ is connected, every vertex has at least one edge incident to it, and so has at least one color incident to it, therefore, no diagonal entry of $B$ is equal to $k+1$. Therefore, every diagonal block of $M$ has all nonzero elements. Let $A$ be the adjacency matrix of $G$. We have shown that the pattern of $M$ is equal to the matrix $(A + I) \otimes J$. Since $J$ is $I$ plus the adjacency matrix of $K_{k+1}$, this has $G \strongprod K_{k+1}$ as its pattern.

    Therefore, a symmetric matrix with the same pattern as the adjacency matrix of $G \strongprod K_{k+1}$ has only two eigenvalues, $-1,1$, meaning that $q(G \strongprod K_{k+1})=2$.
\end{proof}

Our next construction will use the adjacency matrix of a complete graph, whose eigenvalues are well known in spectral graph theory.

\begin{lemma}
    \label{lem:complete}
    The adjacency matrix of $K_n$ has eigenvalues $n-1$ and $-1$.
\end{lemma}

\begin{definition}
    Given a hypergraph $\mathcal{H}$, the \emph{representing graph} $G$ has the same vertex set as $\mathcal{H}$ and an edge connecting any two vertices which are in a hyperedge together. 
\end{definition}

If a hypergraph $\mathcal{H}$ is $l$-uniform, meaning every edge is of size $l$, then the edges of the representing graph will be covered by cliques of size $l$. If $\mathcal{H}$ is also linear, meaning that every pair of edges intersects in at most one vertex, then this covering by cliques will be a partition. 

We want to color the hyperedges of $\mathcal{H}$. Each color class of the edges will correspond to a subgraph of the representing graph $G$ which consists of disjoint copies of $K_l$. We can then use the fact that $K_l$ has only two eigenvalues to do a similar trick to Theorem \ref{thm:maxdeg} and show that $q(G \strongprod K_k)$ is 2 (where $k$ is the chromatic index of $\mathcal{H}$).

\begin{theorem}
    \label{thm:complete}
    Let $G$ be the representing graph of a $l$-uniform linear hypergraph $\mathcal{H}$ with maximum degree $k$. Let $\chi'(\mathcal{H}) = c$. Then:
    \begin{enumerate}[(a)]
        \item \label{case:1} If $c > k$, then $q(G \strongprod K_c)=2.$ 
        \item \label{case:2} If $c = k$, then $q(G \strongprod K_{c+1})=2$. 
        \item \label{case:3} If $\mathcal{H}$ is $k$-regular and $c = k$, then $q(G \modstrong K_c)=2$.
    \end{enumerate}
\end{theorem}
\begin{proof}
    In cases \ref{case:1} and \ref{case:2}, let $A_i$ be:
    \[ A_{i,a,b} = \begin{cases} 1 & \text{ if $a = b$ and $a$ is not incident to a hyperedge of color $i$} \\ 1 & \text{ if $l$ and $m$ share a hyperedge of color $i$} \\ 0 & \text{ otherwise} \end{cases} \]
    Then defining:
    \[ M = (A_1 \otimes J_1) + \cdots + (A_k \otimes J_k) \]
    We can see the eigenvalues of $M$ will be $-1$ and $l-1$ from Lemma \ref{lem:eigenvals} and \ref{lem:complete}. We can also see that the pattern of $M$ will match the pattern of the appropriate strong product by Lemma \ref{lem:blocks}. The only difference between the two cases is that in case \ref{case:2} we first construct a coloring on $c+1$ colors to ensure that at least one color misses each vertex. 

    In case \ref{case:3}, let $A_i$ be:
    \[ A_{i,a,b} = \begin{cases} 1 & \text{ if $l$ and $m$ share a hyperedge of color $i$} \\ 0 & \text{ otherwise} \end{cases} \]
    Then defining:
    \[ M = (A_1 \otimes J_1) + \cdots + (A_k \otimes J_k) \]
    We can see the eigenvalues of $M$ will be $-1$ and $l-1$ from Lemma \ref{lem:eigenvals} and \ref{lem:complete}. We can also see that the pattern of $M$ will match the pattern of $G \modstrong K_c$ by Lemma \ref{lem:blocks}.

    Therefore, we see in all cases that the appropriate graphs have $q$ equal to 2.
\end{proof}

See Figure \ref{fig:examples} for some examples of graphs which this theorem applies to.

\begin{figure}[ht]
    \centering
    \begin{subfigure}{0.4\textwidth}
        \centering
        \begin{tikzpicture}
            [
                vert/.style={circle,fill=white,draw=black},
                edge/.style={thick},
            ]
            \node[vert] (A) at (0,0) {};
            \node[vert] (B1) at (1.7,1) {};
            \node[vert] (B2) at (1.7,-1) {};
            \node[vert] (C1) at (-1.7,1) {};
            \node[vert] (C2) at (-1.7,-1) {};
            \draw[edge] (A) -- (B1) -- (B2) -- (A);
            \draw[edge] (A) -- (C1) -- (C2) -- (A);
        \end{tikzpicture}
        \caption{Bowtie $\strongprod K_2$}
    \end{subfigure}
    \hfill
    \begin{subfigure}{0.4\textwidth}
        \centering
        \begin{tikzpicture}
            [
                vert/.style={circle,fill=white,draw=black},
                edge/.style={thick},
            ]
            \node[vert] (A) at (1*360/6-360/12:1) {};
            \node[vert] (B) at (3*360/6-360/12:1) {};
            \node[vert] (C) at (5*360/6-360/12:1) {};
            \node[vert] (D) at (0*360/6-360/12:3) {};
            \node[vert] (E) at (2*360/6-360/12:3) {};
            \node[vert] (F) at (4*360/6-360/12:3) {};
            \draw[edge] (A) -- (B) -- (C) -- (A);
            \draw[edge] (D) -- (E) -- (F) -- (D);
            \draw[edge] (A) -- (D) -- (C) -- (F) -- (B) -- (E) -- (A);
        \end{tikzpicture}
        \caption{Octahedron $\strongprod K_4$}
    \end{subfigure}
    \hfill
    \begin{subfigure}{0.4\textwidth}
        \centering
        \begin{tikzpicture}
            [
                vert/.style={circle,fill=white,draw=black},
                edge/.style={thick},
            ]
            \node[vert] (A1) at (0,0) {};
            \node[vert] (A2) at (0,2) {};
            \node[vert] (A3) at (0,4) {};
            \node[vert] (B1) at (2,0) {};
            \node[vert] (B2) at (2,2) {};
            \node[vert] (B3) at (2,4) {};
            \node[vert] (C1) at (4,0) {};
            \node[vert] (C2) at (4,2) {};
            \node[vert] (C3) at (4,4) {};
            \draw[edge] (A1) -- (A2) -- (A3) edge[bend left] (A1);
            \draw[edge] (B1) -- (B2) -- (B3) edge[bend left] (B1);
            \draw[edge] (C1) -- (C2) -- (C3) edge[bend left] (C1);
            \draw[edge] (A1) -- (B1) -- (C1) edge[bend left] (A1);
            \draw[edge] (A2) -- (B2) -- (C2) edge[bend left] (A2);
            \draw[edge] (A3) -- (B3) -- (C3) edge[bend left] (A3);
        \end{tikzpicture}
        \caption{$(K_3 \boxprod K_3) \modstrong K_2$ and $(K_3 \boxprod K_3) \strongprod K_3$}
        \label{fig:cartesian}
    \end{subfigure}
    \caption{Some examples of graphs where Theorem \ref{thm:complete} applies.}
    \label{fig:examples}
\end{figure}
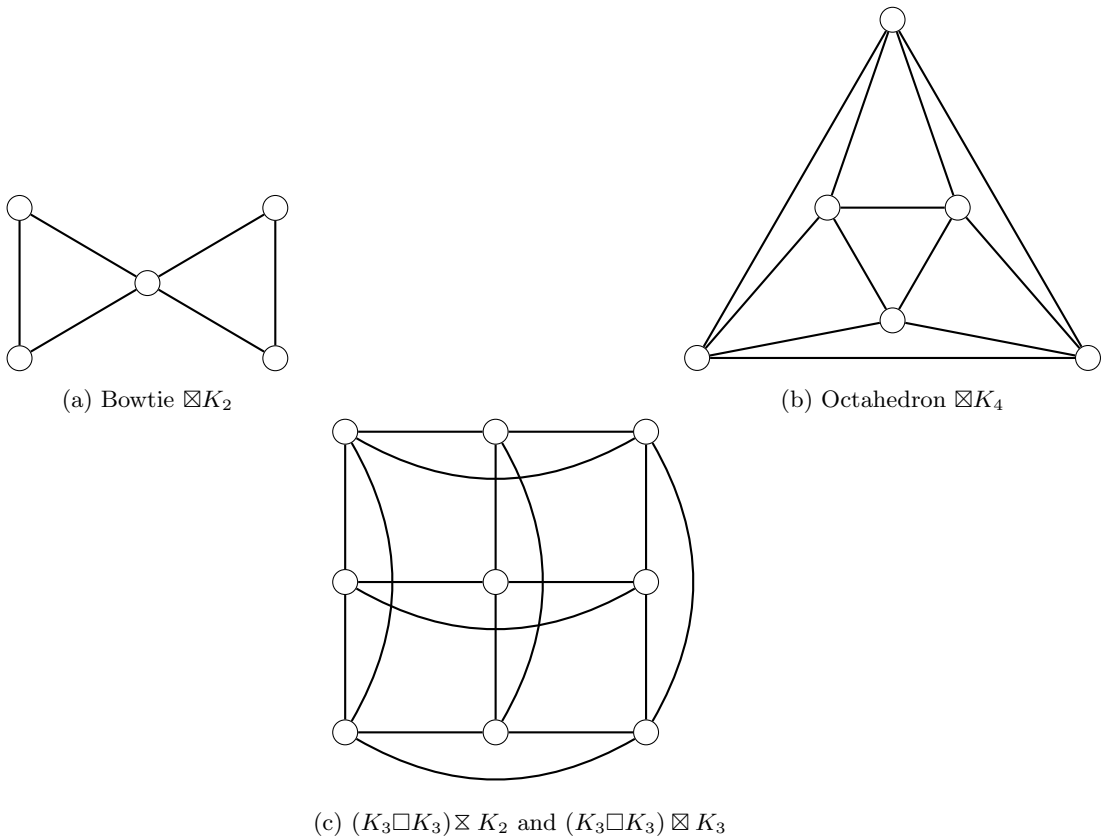

Figure \ref{fig:cartesian} is particularly illustrative of the power of this theorem. Any graph $K_l \boxprod K_l$ can have its edges partitioned into two spanning subgraphs whose components are each $K_l$. Therefore, by the theorem, $q((K_l \boxprod K_l) \modstrong K_2)=2$ for all $l$. 
We can also use Theorem \ref{thm:complete} case \ref{case:2} to conclude that $q((K_l \boxprod K_l) \strongprod K_3)=2$.

Finally, we remark that work in \cite{barrett2023sparsity} describes a family of graphs which they call \emph{double ended candles} which they show satisfies $q=2$, and they remark that duplicating the end vertices yields a graph that also has $q=2$. We have pictured these in Figure \ref{fig:candle}. We note that this last family can be realized simply as $P_k\modstrong K_2$ which will have $q=2$ by Theorem \ref{thm:onefactor} since the chromatic index of $P_k$ is 2.

Similarly, the paper \cite{barrett2024graphs} describes the \emph{closed candles} pictured in Figure \ref{fig:closed} and prove that they satisfy $q=2$.  We note that these arise simply as $C_k\modstrong K_2$, and so Theorem \ref{thm:complete} gives another proof that these satisfy $q=2$.

\begin{figure}
    \begin{subfigure}{0.4\textwidth}
          \centering
     \begin{tikzpicture}
        [
                vert/.style={circle,fill=white,draw=black},
                 edge/.style={thick},
             ]
          \draw[edge]   (0,0)node[vert]{}--(0,1)node[vert]{}--(0,2)node[vert]{}--(0,3)node[vert]{}(1,0)node[vert]{}--(1,1)node[vert]{}--(1,2)node[vert]{}--(1,3)node[vert]{} (0,0)--(1,1)--(0,2)--(1,3) (1,0)--(0,1)--(1,2)--(0,3);
         \end{tikzpicture}
         \caption{$q(P_k\modstrong K_2)=2$}\label{fig:candle}
    \end{subfigure}
     \begin{subfigure}{0.4\textwidth}
          \centering
        \begin{tikzpicture}
        [
                vert/.style={circle,fill=white,draw=black},
                 edge/.style={thick},
             ]
          \draw[edge]   (0:1)node[vert]{}--(60:1)node[vert]{}--(120:1)node[vert]{}--(180:1)node[vert]{}--(240:1)node[vert]{}--(300:1)node[vert]{}--(0:1)
          (0:2)node[vert]{}--(60:2)node[vert]{}--(120:2)node[vert]{}--(180:2)node[vert]{}--(240:2)node[vert]{}--(300:2)node[vert]{}--(0:2)
          (0:1)--(60:2)--(120:1)--(180:2)--(240:1)--(300:2)--(0:1)
          (0:2)--(60:1)--(120:2)--(180:1)--(240:2)--(300:1)--(0:2)
          ;
         \end{tikzpicture}
         \caption{$q(C_k\modstrong K_2)=2$}\label{fig:closed}
    \end{subfigure}
 
     \caption{Modified strong products of cycles and paths with $K_2$. }
     \label{fig:candles}
 \end{figure}
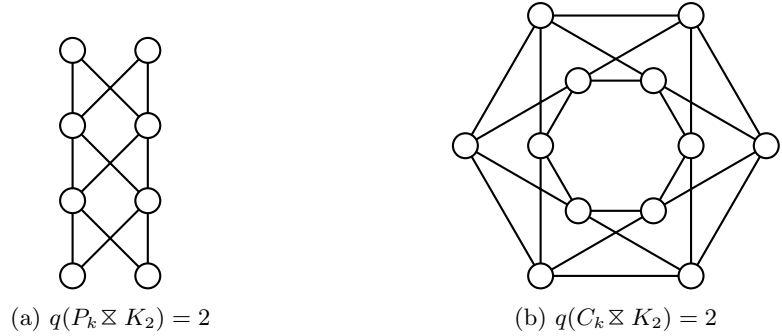

\section{Conclusion}
We have given several constructions of families of graphs that can achieve only two distinct eigenvalues.  We have done this primarily by way of Lemmas \ref{lem:eigenvals} and \ref{lem:blocks}.  In all of our constructions, we have somehow partitioned the edge set into pieces where we can write down 0-1 matrices for each piece that all have the same two eigenvalues.  It would be of interest to see how far this idea can be pushed.  Lemma 2.3 of \cite{ahmadi2013minimum} implies that any graph that can achieve two distinct eigenvalues can achieve \emph{any} two eigenvalues that we wish.  So any time we can partition the edge set into pieces $G_i$ where each $q(G_i)=2$, then we could weight the edges appropriately to make sure each subgraph had the same two eigenvalues.  Then taking the appropriate tensor products, Lemma \ref{lem:eigenvals} would yield a matrix for a new graph that would have exactly two distinct eigenvalues.  It is not clear, however, what the pattern of this new graph would be.  Since we have weighted edges, we cannot simply apply Lemma \ref{lem:blocks} to determine the pattern.  It would be of interest to characterize all graphs that arise in this way. It would also be of interest if there are other ``product-like" constructions that can give more examples that achieve two distinct eigenvalues.

\bigskip

\textbf{Acknowledgment:} The authors would like to thank Wayne Barrett for helpful discussions, especially around Lemma \ref{lem:existence}.

\pagebreak

\bibliographystyle{plain}
\bibliography{mybibfile}
\end{document}